\documentclass{article}
\usepackage[utf8]{inputenc}
\usepackage{amssymb}
\usepackage{amsmath}
\usepackage{amsthm}
\usepackage{tikz}
\usepackage{url}
\usepackage{xcolor}
\usepackage{mathtools}
\usepackage{geometry}
\usepackage{enumerate}
\usepackage{accents}
\usepackage{hyperref}
\usepackage{xkeyval}
\usepackage[affil-it]{authblk}
\usepackage{stmaryrd}
\usepackage{comment}

\usetikzlibrary{calc}

\newcommand{\range}{\operatorname{ran}}

\newcommand{\append}{{}^\frown}

\newcommand{\boldpi}{\boldsymbol{\Pi}}

\newcommand\forces{\Vdash}

\newcommand{\Pow}{\mathcal{P}}

\newcommand{\meager}{\mathcal{M}}

\newcommand{\stem}{\operatorname{stem}}
\newcommand{\suc}{\operatorname{succ}}

\newcommand{\restrict}{\upharpoonright}

\newcommand{\Mathias}{\mathbb{M}}
\newcommand{\Laver}{\mathbb{L}}
\newcommand{\Miller}{\mathbb{MI}}
\newcommand{\Hechler}{\mathbb{D}}
\newcommand{\Sacks}{\mathbb{S}}
\newcommand{\ultU}{\mathcal{U}}
\newcommand{\pred}{\operatorname{pred}}
\newcommand{\name}[1]{\underaccent{\tilde}{#1}}

\newcommand{\subseq}{\subseteq}
\newcommand{\subseqneq}{\subsetneq}
\newcommand{\splt}{\operatorname{sp}}

\newcommand{\seq}[1]{{\langle#1\rangle}}
\DeclarePairedDelimiter\abs{\lvert}{\rvert}

\renewcommand\emptyset{\varnothing}
\renewcommand\subset{\subseteq}
\renewcommand{\setminus}{\smallsetminus}

\renewcommand{\le}{\leqslant}
\renewcommand{\ge}{\geqslant}
\renewcommand{\leq}{\leqslant}

\newcommand{\onetotwo}{i\in \{1,\dots, 2^{n+1}\}}
\newcommand{\onetotwoy}[1]{#1\in \{1,\dots, 2^{n+2}\}}
\newcommand{\onetotwox}{\onetotwoy{i}}

\newcommand{\on}{\mathord\upharpoonright}

\makeatletter
\define@cmdkeys{constellation}[]{minimal,notunbdd,notdominating,dminimal,bminimal}
\makeatother
\presetkeys{constellation}{}{minimal=0,notunbdd=0,notdominating=0,dminimal=0,bminimal=0}
\newcommand\retrievecolor[1]{\ifnum#1=1 black\else white\fi}
\newcommand\diagram[1]{%
	\begingroup%
	\setkeys{constellation}{#1}%
    \[%
	\begin{tikzpicture}[scale=0.3]
		\filldraw[fill/.expand once=\retrievecolor{\minimal},draw=black] (0,0) rectangle (1,1);
		\filldraw[fill/.expand once=\retrievecolor{\dminimal},draw=black] (5,0) rectangle (6,1);
		\filldraw[fill/.expand once=\retrievecolor{\bminimal},draw=black] (10,0) rectangle (11,1);
		\filldraw[fill/.expand once=\retrievecolor{\notunbdd},draw=black] (5,5) rectangle (6,6);
		\filldraw[fill/.expand once=\retrievecolor{\notdominating},draw=black] (10,5) rectangle (11,6);
		\draw[->] (1+0.5,0.5) -- (5-0.5,0.5);
		\draw[->] (5+1+0.5,0.5) -- (10-0.5,0.5);
		\draw[->] (5+1+0.5,5.5) -- (10-0.5,5.5);
		\draw[->] (5+0.5,5-0.5) -- (5+0.5,1.5);
		\draw[->] (10+0.5,5-0.5) -- (10+0.5,1.5);
	\end{tikzpicture}%
    \]%
	\endgroup%
}

\newcommand{\subjclass}[2][2020]{%
  \let\oldthefootnote\thefootnote%
  \renewcommand{\thefootnote}{}%
  \footnotetext{\emph{2020 Mathematics Subject Classification.} #2}%
  \let\thefootnote\oldthefootnote%
}
\newcommand{\keywords}[1]{%
  \let\oldthefootnote\thefootnote%
  \renewcommand{\thefootnote}{}%
  \footnotetext{\emph{Key words and phrases.} #1}%
  \let\thefootnote\oldthefootnote%
}

\theoremstyle{definition}
\newtheorem{thm}{Theorem}[section]
\newtheorem*{thm*}{Theorem}
\newtheorem{defi}[thm]{Definition}
\newtheorem*{defi*}{Definition}
\newtheorem{lem}[thm]{Lemma}

\newtheorem*{lem*}{Lemma}
\newtheorem{fact}[thm]{Fact}
\newtheorem*{fact*}{Fact}

\newtheorem*{rmk*}{Remark}
\newtheorem{cor}[thm]{Corollary}
\newtheorem*{cor*}{Corollary}
\newtheorem*{convention*}{Convention}
\newtheorem*{notation*}{Notation}
\newtheorem{question}[thm]{Question}
\newtheorem{notation}[thm]{Notation}

\usepackage[backend=biber,style=alphabetic,sorting=nty,doi=false,isbn=false,url=false,eprint=false]{biblatex}
\renewbibmacro{in:}{}

\definecolor{myred}{RGB}{230, 150, 100}
\definecolor{mygreen}{RGB}{150, 230, 100}
\definecolor{myblue}{RGB}{130, 220, 220}
\hypersetup{linkbordercolor=myred,citebordercolor=mygreen,urlbordercolor=myblue}

\title{Two variants of minimality in forcing extensions
}
\author{Martin Goldstern}

\affil{Institute of Discrete Mathematics and Geometry, TU Wien, Wiedner Hauptstra\ss{}e 8-10/104, 1040 Wien, Austria. \\e-mail: martin.goldstern@tuwien.ac.at}
\author{Tatsuya Goto}

\affil{Institute of Discrete Mathematics and Geometry, TU Wien, Wiedner Hauptstra\ss{}e 8-10/104, 1040 Wien, Austria.
\\e-mail: goto.tatsuya@icloud.com}
\date{\today}

\begin{document}
	\maketitle
    

    \begin{abstract}
        We introduce d-minimal and b-minimal extensions, two weakenings of minimality based respectively on eventual domination and infinitely often domination. We prove that the generic extensions obtained by the full Mathias forcing and by Mathias forcing relative to Ramsey ultrafilters are d-minimal, whereas Hechler extensions are b-minimal. We also construct a d-minimal extension that is weakly $\omega^\omega$-bounding but neither minimal nor $\omega^\omega$-bounding, and realize every constellation of these properties compatible with their natural implications.
    \end{abstract}

    \section{Introduction}

    A fundamental question in the study of forcing extensions is to what extent the entire extension can be recovered from a single new real belonging to it. Suppose that $W$ is a forcing extension of $V$. We say that $W$ is minimal over $V$ if $V[x]=W$
    for every real $x\in W\setminus V$. Thus, every new real in the extension recovers the entire generic extension, and there is no proper real-generated intermediate extension between $V$ and $W$. Sacks, Laver, and Miller extensions are known to be minimal \cite{Sacks1971PerfectClosedSets, Gray1980IteratedForcing, Miller1984RationalPerfectSet}.
    
    It is natural to ask what remains of this property if, for a real $z\in\omega^\omega$ in the extension, we require a new real $x$ to recover not $z$ itself, but only some information about the growth of $z$. More specifically, even if $z\notin V[x]$, the model $V[x]$ may contain a function that eventually dominates $z$, or at least exceeds $z$ infinitely often. This leads to two weakenings of minimality, which we call d-minimality and b-minimality.
    
    These notions lie naturally between classical minimality and the familiar bounding properties. Every minimal extension is d-minimal, and every d-minimal extension is b-minimal. Moreover, if $W$ is an $\omega^\omega$-bounding extension of $V$, then it is d-minimal, while if $W$ is a weakly $\omega^\omega$-bounding extension of $V$, then it is b-minimal.
    
    In this paper, we study these properties for several classical forcing notions. We show that an extension by Mathias forcing with respect to an ultrafilter $\ultU$ is d-minimal if and only if $\ultU$ is Ramsey (Theorems~\ref{thm:mathiasramsey2} and~\ref{thm:mathiasramseyno}). In particular, the full Mathias extension is d-minimal (Corollary~\ref{cor:mathias}).
    
    We also prove that every Hechler extension is b-minimal (Theorem~\ref{thm:hechlerbmin}). Furthermore, we construct a d-minimal extension that is weakly $\omega^\omega$-bounding but is neither minimal nor $\omega^\omega$-bounding (Section~\ref{sec:example}). Finally, in Section~\ref{sec:constellations}, we show that every combination of minimality, d-minimality, b-minimality, $\omega^\omega$-bounding, and weak $\omega^\omega$-bounding that is consistent with the natural implications among these properties can be realized.
    
\begin{notation}
\begin{itemize}
    \item For $x, y \in \omega^\omega$: 
     $x <^* y$ means $y$ almost dominates $x$, i.e., $\forall^\infty n : x(n)<y(n)$.
     \item For $x, y \in \omega^\omega$:
    $x <^\infty y$ means $y$ is infinitely often above $x$: $\exists^\infty n: x(n)<y(n)$.
    \item $\omega^{<\uparrow \omega}$ is the set of finite strictly increasing sequences. 
    \item 
    We let $2^{\le \omega}:= 2^{<\omega} \cup 2^\omega$. \\
	For $x,y\in 2^{\mathord\le\omega} $, let $\delta(x,y)$ be the smallest $i$  such that $x(i),y(i)$  are both defined but not equal, and $:=-1$ if there is no such $i$ (i.e., if $x=y$, or one of $x,y$ is an initial segment of the other). 
    \end{itemize}

\end{notation}


    \begin{defi}
        Let $W$ be a forcing extension of $V$.
        In $W$, we define the following notions.
        \begin{itemize}
            \item $W$ is a d-minimal extension of $V$ iff:\\
            for every $z \in \omega^\omega$ and every $x \in 2^\omega \setminus V$, there is $y \in V[x] \cap \omega^\omega$ such that $z <^* y$.
            \item $W$ is a b-minimal extension of $V$ iff:\\
            for every $z \in \omega^\omega$ and every $x \in 2^\omega \setminus V$, there is $y \in V[x] \cap \omega^\omega$ such that $z <^\infty y$.
        \end{itemize}
    \end{defi}

    In other words, $W$ is a d-minimal extension of $V$ iff $W$ is an $\omega^\omega$-bounding extension over $V[x]$ for every new real $x$.
    Also, $W$ is a b-minimal extension of $V$ iff $W$ is a weakly $\omega^\omega$-bounding extension over $V[x]$ for every new real $x$.
    
%
%
%

    \section{Mathias extensions}\label{sec:mathias}

    In this section, we prove Mathias forcing yields d-minimal extensions.

    Let $\Mathias$ denote the full Mathias forcing and $\Mathias_\ultU$ denote the Mathias forcing relative to an ultrafilter $\ultU$ on $\omega$. Also $\name{m}$ is the name of the generic real with respect to $\Mathias$ or $\Mathias_\ultU$.

    We sometimes ignore the requirement $\max s < \min A$ for a Mathias condition $(s,A)$. When we do so, we use $(s,A)$ to mean $(s, A \setminus (\max s + 1))$.

    \begin{defi}
        Let $\ultU$ be an ultrafilter on $\omega$.
        A $\ultU$-tree is a subtree $T$ of $\omega^{<\omega}$ such that $\suc_T(s) \in \ultU$ for every $s \in T$.
        For a subtree $T$ of $\omega^{<\omega}$, we say $p$ is a $\ultU$-branch if it is a branch through $T$ with $\range(p) \in \ultU$.
    \end{defi}

    \begin{lem}[\cite{grigorieff1971combinatorics}]\label{lem:griorieff}
        Let $\ultU$ be an ultrafilter on $\omega$.
        Then the following are equivalent.
        \begin{enumerate}
            \item $\ultU$ is Ramsey.
            \item Every $\ultU$-tree has a $\ultU$-branch. \qed
        \end{enumerate}
    \end{lem}

    For a while, $\ultU$ is a fixed Ramsey ultrafilter.
	\begin{lem}\label{lem:consequenceofpuredecision}
		Let $(s, A) \in \Mathias_\ultU$ and $\name{x}$ be a $\Mathias_\ultU$-name for a real in $2^\omega$.
		Then there are $B \subset A$ in $\ultU$ and $\seq{x_{\bar{n}} : \bar{n} \in \omega^{<\uparrow \omega}}$ such that
		\begin{enumerate}[(A)]
			\item $(s \append b_{n_1} \append \dots \append b_{n_l}, B) \forces \name{x} \restrict (b_{n_l} + 1) = x_{n_1, \dots, n_l}$ for every nonempty $\seq{n_1, \dots, n_l} \in \omega^{<\uparrow \omega}$.
		\end{enumerate}
		Here, $B = \{ b_0 < b_1 < b_2 < \dots \}$.
	\end{lem}
	\begin{proof}
        We define a $\ultU$-tree $T$ and a family $\seq{x_{n_1, \dots, n_l} : \seq{n_1, \dots, n_l} \in \omega^{<\uparrow \omega}}$ as follows.
        We put $\emptyset$ into $T$.
        Suppose $t \in T$.
        Look at condition $(s \append t, A_{\pred(t)})$.
        By the pure decision property of $\Mathias_\ultU$, there is $A_t \subset A_{\pred(t)}$ in $\ultU$ such that $\max (s \append t) < \min (A_t)$ and $(s \append t, A_t)$ decides $\name{x} \restrict (\max(t) + 1)$ (If $s = \emptyset$, then we use $A$ instead of $A_{\pred(t)}$).
        Call the decided value $y_t$.
        Let $\suc_T(t) := A_t$.

        Then by Lemma \ref{lem:griorieff}, we can find a $\ultU$-branch $p$ of $T$.
        Let $B := \range(p) \in \ultU$.
        Then clearly if $t \in p$, then $B \subset A_t$.
        Let $B = \{ b_0 < b_1 < b_2 < \dots \}$.
        For a sequence $\seq{n_1, \dots, n_l} \in \omega^{<\uparrow \omega}$, we let $x_{n_1, \dots, n_l} := y_{b_{n_1},\dots,b_{n_l}}$.
        Then condition (A) is clear because $B \subset A_\seq{{b_{n_1}, \dots, b_{n_l}}}$.
	\end{proof}

    \begin{lem}\label{lem:seqcptwrtppt}
        Let $\ultU$ be a P-point and $X$ be a compact, first countable space.
        Let $\seq{x_n : n \in \omega }$ be a sequence of points in $X$.
        Then there is a set $A \in \ultU$ such that the sequence $\seq{ x_n : n \in A }$ converges to some point.
    \end{lem}
    \begin{proof}
        Since $X$ is compact, there is a $\ultU$-limit $y \in X$ of the sequence.
        By using first countability, we can take a neighborhood basis $\seq{U_i : i \in \omega}$ of the point $y$.
        For each $i \in \omega$, fix $A_i \in \ultU$ such that $\{x_n : n \in A_i\} \subset U_i$.
        Then, a pseudo-intersection of $\seq{A_i : i \in \omega}$ in $\ultU$ is a desired member.
    \end{proof}

	\begin{lem}\label{lem:mainlemma}
		Let $(s, A) \in \Mathias_\ultU$ and $\name{x}$ be a $\Mathias_\ultU$-name for a real in $2^\omega$.
		Then there are 
        \begin{itemize}
            \item 
        $B = \{ b_0 < b_1 < b_2 < \dots \} \subset A$ in $\ultU$,
        \item $\seq{x_{\bar{n}} : \bar{n} \in \omega^{<\uparrow \omega}}$, each $x_{\bar n}$ in $2^{<\omega}$,
        \item and $\seq{y_{\bar{n}} : \bar{n} \in \omega^{<\uparrow \omega}}$, each $y_{\bar{n}}$ in $2^\omega$,
        \end{itemize}
        such that
		\begin{enumerate}[(A)]
			\item $(s \append b_{n_1} \append \dots \append b_{n_l}, B) \forces \name{x} \restrict (b_{n_l} + 1) = x_{n_1, \dots, n_l}$ for every nonempty $\seq{n_1, \dots, n_l} \in \omega^{<\uparrow \omega}$.
			\item $\lim_{n_l \to \infty} (x_{\bar{n}, n_l} \append 0^\omega) = y_{\bar{n}}$  for every $\bar{n} = \seq{n_1, \dots, n_{l-1}} \in \omega^{<\uparrow \omega}$.
			\item for each $\bar{n} = \seq{n_1, \dots, n_{l}} \in \omega^{<\uparrow \omega}$, we have either:
			\begin{itemize}
				\item for all $j < j'$ (above $n_l$) the finite sequences $x_{\bar{n},j}$ and $x_{\bar{n},j'}$ are incompatible, or
				\item for all $j < j'$ (above $n_l$) the finite sequences $x_{\bar{n},j}$ and $x_{\bar{n},j'}$ are compatible, i.e. $x_{\bar{n},j} \subset x_{\bar{n},j'}$.
			\end{itemize}
			\item Moreover, in the first case we have for all $j < j'$, $\delta(x_{\bar{n}, j'}, y_{\bar{n}}) > b_j$.
		\end{enumerate}
	\end{lem}
	\begin{proof}
		To satisfy (A), we use just Lemma \ref{lem:consequenceofpuredecision}.
		Next, to satisfy (B), we use Lemma \ref{lem:seqcptwrtppt} repeatedly and use Lemma \ref{lem:griorieff}.
		To satisfy (C), we use Ramseyness of $\ultU$ repeatedly and use Lemma \ref{lem:griorieff}.
		Finally, to satisfy (D), we remove all counterexamples $j'$ repeatedly and use Lemma \ref{lem:griorieff}.
	\end{proof}

	Fix $(s, A) \in \Mathias_\ultU$.
	Fix $B = \{ b_0 < b_1 < b_2 < \dots \} \subset A$ in $\ultU$, $\seq{x_{\bar{n}} : \bar{n} \in \omega^{<\uparrow \omega}}$ and $\seq{y_{\bar{n}} : \bar{n} \in \omega^{<\uparrow \omega}}$ obtained by applying Lemma \ref{lem:mainlemma} to the condition $(s, A)$.
	Define a name $\name k_{\bar n}$ by $\forces \name k_{\bar n} = \delta(\name x, y_{\bar n})$.
	Also define names $\name i_{\bar{n}}$ and $\name j_{\bar{n}}$ by $\forces \name i_{\bar{n}} = \min(i > \max(\bar{n
    )}:  \delta(x_{\bar n,i}, y_{\bar{n}}) = k_{\bar{n}} )  $ (and $:=-1$ if this is not defined) and $\forces \name j _{\bar{n}} = b_{\name i_{\bar{n}}}$.
	Finally, let $\forces \name \ell(\bar n):=\max(\name k_{\bar n}, \name j_{\bar n})$. 
	
	\begin{lem}
		Assume $\forces \name{x} \not \in V$.\\
		Then for all $\bar{n}$, we have $(s \append \seq{b_{n_1}, \dots, b_{n_l}}, B) \forces \name{m}(\abs{s} + \abs{\bar{n}}) \le \ell_{\bar{n}}$.
	\end{lem}
	\begin{proof}
		
		Fix $\bar n$ and let $\bar{b} := s \append \seq{b_{n_1}, \dots, b_{n_l}}$.  We consider two cases, according to what
		happens in condition (C). 
		
		\begin{description}
			\item[Case 1] All the values $x_{\bar n, i}$ extend  each other.
            So we get $x_{\bar n, i} \subseteq y_{\bar n}$ for all $i \in B$. 
            Now note that it is forced that $\name{x} \ne y_{\bar n}$ since $\forces \name{x} \not \in V$.
			We have 
			$$ (\bar b \append b_{n^*}, B)  \forces \name x \restrict (b_{n^*} + 1) = x_{\bar n \append n^*} \subseteq y_{\bar n}, \text{ so } k_{\bar{n}} \ge b_{n^*} = \name m(\abs{s} + |\bar{n}|).$$
			
			As this is true for all $b$, we get $(\bar b, B) \forces \name m(\abs{s} + |\bar{n}|) \le \name k_{\bar{n}}$. 
			
			\item[Case 2] All the values $x_{\bar n, i}$ disagree with  each other.
			Then for all $n^*$ we have 
			$$ (\bar b \append b_{n^*}, B) \forces \name x \supseteq x_{\bar n \append n^*}, \text{ so } k_{\bar n} = \delta(\name x, y_{\bar n}) = \delta(x_{\bar n \append n^*}, y_{\bar{n}})$$
			Note that now there is a unique $i$ such that $\delta(x_{\bar n^\frown i}, y_{\bar n}) = k_{\bar{n}}$, because all the $x_{\bar n, i}$ diverge at different points from $y_{\bar{n}}$. 
			So we also have 
			$$ (\bar b \append b_{n^*}, B) \forces  \name m(\abs{s} + |\bar n|) = b_{n^*} = j_{\bar n},$$
			hence $(\bar b, B) \forces \name m(\abs{s} + |\bar{n}|) = j_{\bar n}$.  \qedhere
		\end{description}
	\end{proof}

	\begin{thm}\label{thm:mathiasramsey}
		Let $\ultU$ be a Ramsey ultrafilter.
        Then $\Mathias_\ultU \forces \forall x \in 2^\omega \setminus V\ \exists y \in V[x] \cap \omega^\omega\ \name{m} \le^* y$.
	\end{thm}
	\begin{proof}
		Fix a condition $(s, A) \in \Mathias_\ultU$.
		Let $\name{x}$ be a $\Mathias_\ultU$-name for a real in $2^\omega \setminus V$.
		Take a $B$ by the above lemma.
		We have $(s, B) \forces \name{m}(|s|) \le \ell_{\emptyset}$.
		  Let $\forces \name{y}(|s|) :=\name \ell_{\emptyset}$.
		Next we let $\forces \name{y}(|s|+1) := \max \{ \name \ell_{\seq{a}} : a \le \name{y}(|s|) \}$. Then we have $(s, B) \forces \name{m}(|s|+1) \le y(|s|+1)$.
		
		Continuing this process, we can get a real $y \in \omega^\omega \cap V[x]$ that dominates the Mathias real $m$.
	\end{proof}

    \begin{lem}\label{lem:crn_mathias}
        Let $\ultU$ be a Ramsey ultrafilter. Then $\Mathias_\ultU$ has continuous reading of names. More precisely, for every
    $p\in\Mathias_\ultU$ and every $\Mathias_\ultU$-name $\name z$ for an
    element of $\omega^\omega$, there are $q\leq p$ and a continuous function
    \[
        F\colon [q]\longrightarrow\omega^\omega
    \]
    coded in $V$ such that
    \[
        q\forces_{\Mathias_\ultU}\name z=F(\name m),
    \]
    \end{lem}
    
    \begin{proof}
    By \cite[Theorems~1.19 and~1.20]{JudahShelah1989}, since $\ultU$ is
    Ramsey, $\Mathias_\ultU$ can be identified with the dense subforcing
    of $\Laver_\ultU$ consisting of the trees associated with Mathias
    conditions. Under this identification, the Mathias real corresponds
    to the generic branch. We may therefore regard $p$ as a condition
    $T_p\in\Laver_\ultU$ and $\name z$ as a $\Laver_\ultU$-name.
    
    For each $n\in\omega$, let
    \[
        D_n=
        \{T\in\Laver_\ultU:T\text{ decides the value of }\name z(n)\}.
    \]
    Each $D_n$ is open dense. By
    \cite[Lemma~1.6(b)]{JudahShelah1989}, there is a pure extension
    $S\leq^0 T_p$ such that, for every $n\in\omega$, the set $\{t\in S:S_t\in D_n\}$
    contains a front $H_n$ of $S$. For each $t\in H_n$, let $k_{n,t}$ be
    the unique integer such that
    \[
        S_t\forces_{\Laver_\ultU}\name z(n)=k_{n,t}.
    \]
    
    Every branch $a\in[S]$ passes through a unique member $t_n(a)$ of
    $H_n$. Define
    \[
        F_S(a)(n)=k_{n,t_n(a)}.
    \]
    For each fixed $n$, the value $F_S(a)(n)$ is determined by the finite
    initial segment $t_n(a)$ of $a$. Hence every coordinate of $F_S$ is
    locally constant, and therefore
    \[
        F_S\colon[S]\to\omega^\omega
    \]
    is continuous. Moreover, since
    $\{S_t:t\in H_n\}$
    is predense below $S$, we have
    \[
        S\forces_{\Laver_\ultU}
        \name z(n)=F_S(\name m)(n)
    \]
    for every $n\in\omega$. Consequently,
    \[
        S\forces_{\Laver_\ultU}\name z=F_S(\name m).
    \]
    
    Finally, by the density of the Mathias trees in $\Laver_\ultU$, choose
    $q\leq p$ in $\Mathias_\ultU$ such that its associated tree $T_q$ is
    contained in $S$. Then
    $F=F_S\mathbin{\upharpoonright}[q]$
    is continuous and
    \[
        q\forces_{\Mathias_\ultU}\name z=F(\name m). \qedhere
    \]
    \end{proof}

	\begin{thm}\label{thm:mathiasramsey2}
		Let $\ultU$ be a Ramsey ultrafilter. 
        Then $\Mathias_\ultU$ forces that $ V[\name{m}]$ is a d-minimal extension over $V$.
	\end{thm}
    \begin{proof}
        Let $m$ be the $\Mathias_\ultU$-generic real and work in $V[m]$.
        Let $z \in \omega^\omega$ and $x \in 2^\omega \setminus V$.
        By the continuous reading of names, we can find a continuous function $f \colon \omega^\omega \to \omega^\omega$ coded in $V$ such that $z = f(m)$.
        By Theorem \ref{thm:mathiasramsey}, we can take $y \in V[x] \cap \omega^\omega$ such that $m < y$.
        Since $\{ f(w) : w \le y \}$ is a compact set coded in $V[x]$, we can take $y' \in V[x] \cap \omega^\omega$ that is an upper bound of this set.
        Then we have $z = f(m) < y'$.
    \end{proof}

    \begin{lem}\label{lem:elimination_of_U}
		Let $\ultU$ be a $(V, \Pow(\omega) / \mathsf{fin})$-generic ultrafilter and $m$ be a $(V[\ultU], \Mathias_\ultU)$-generic real.
        Then, in $V[\ultU][m]$, for every real $x \in 2^\omega$, we have $V[\ultU, x] \cap 2^\omega = V[x] \cap 2^\omega$.
    \end{lem}
    \begin{proof}
        Work in $V[\ultU][m]$ and fix $x \in 2^\omega$.
        Let $G$ be the $\Mathias_\ultU$-generic filter corresponding to $m$.
        Let $B = \mathsf{ro}(\Mathias_\ultU)^{V[\ultU]}$.
        Take a $B$-name $\name{x}$ in $V[\ultU]$ such that $\name{x}[G] = x$.
        Let $B_{\name{x}}$ be the complete Boolean subalgebra of $B$ generated by $\llbracket \name{x}(n) = i \rrbracket$ for $n \in \omega$ and $i \in 2$.
        Then we have $V[\ultU][x] = V[\ultU][G \cap B_{\name{x}}]$ (see Lemma 15.40 and Corollary 15.42 of \cite{Jech2003}).

        Fix $y \in V[\ultU, x] \cap 2^\omega$. Since $B_{\name{x}}$ is ccc, there is a Borel function $F$ coded in $V[\ultU]$ such that $y = F(x)$.
        But $V[\ultU]$ and $V$ have the same reals, $F$ is coded already in $V$.
        So $y = F(x) \in V[x]$.
    \end{proof}

	\begin{cor}\label{cor:mathias}
        $\Mathias$ forces that $ V[\name{m}]$ is a d-minimal extension over $V$.
    \end{cor}
    \begin{proof}
        Use Theorem \ref{thm:mathiasramsey2}, the decomposition $\Mathias \simeq \Pow(\omega) / \mathsf{fin} \ast \name{\Mathias}_{\name{\ultU}}$, where $\name{\ultU}$ is the name of the Ramsey ultrafilter added by $\Pow(\omega) / \mathsf{fin}$, and Lemma \ref{lem:elimination_of_U}.

        Another way to prove this corollary is to repeat the same argument used in the proof of Theorem \ref{thm:mathiasramsey}.
    \end{proof}

    The following theorem is the converse of Theorem \ref{thm:mathiasramsey}.
    
    \begin{thm}\label{thm:mathiasramseyno}
        Let $\ultU$ be a non-Ramsey ultrafilter.\\
        Then $\Mathias_\ultU \forces V[\name{m}] \text{ is not a d-minimal extension over } V$.
    \end{thm}
    \begin{proof}
        This proof uses the well-known construction of Cohen reals added by the Mathias forcing relative to a non-Ramsey ultrafilter.
    
        Fix a coloring $f \colon [\omega]^2 \to 2$ witnessing that $\ultU$ is not a Ramsey ultrafilter.
        Define a name $\name{c}$ of a real as follows:
        \[
        \forces \name{c}(n) := f(\name{m}(2n), \name{m}(2n+1)) \text{ for every } n.
        \]
        It can be easily shown that $\forces \name{c} \text{ is a Cohen real over } V$.
        Therefore, the following map $\pi \colon D \to 2^{<\omega}$ is a projection:
        $\pi(s, A) := \seq{f(s(2i), s(2i+1)) : i < |s| / 2}$ for $(s, A) \in D := \{ (s, A) \in \Mathias_\ultU : |s| \text{ is even} \}$.

        We shall show that $\forces \name{m} \text{ is an unbounded real over } V[c]$.

        Let $m$ be $\Mathias_\ultU$-generic real over $V$ and let $c := \name{c}[m]$. Work in $V[c]$.
        The quotient forcing $Q$ that satisfies $\pi[D] \ast Q \simeq \Mathias_\ultU$ can be defined as follows:
        \[
        Q := \{ (s, A) \in D : \pi(s, A) \subset c \}.
        \]
        Let $y \in \omega^\omega$, $(s, A) \in Q$ and $N \in \omega$.
        We may assume that $n := |s| > N$.
        Since $f$ witnesses that $\ultU$ is not a Ramsey ultrafilter, we can take $i, j \in A$ such that $y(n) < i < j$ and $f(i, j) = c(n / 2)$.
        Then $(s \append i \append j, A \setminus (j+1)) \in Q$ and it forces that $\name{m}(n) = i > y(n)$.
    \end{proof}

    \section{Hechler extensions}\label{sec:hechler}

    In this section, we prove that Hechler forcing yields b-minimal extensions. Let $\Hechler$ denote the Hechler forcing and $\name{d}$ denote the name of the Hechler real. Also let $\meager$ denote the meager ideal on the Cantor space.
    Also, the statement $(\meager \cap V) \not \in \meager$ is a shorthand to say that there is a meager set containing all meager sets coded in $V$.

    \begin{fact}\label{fact:hechler}
        \begin{enumerate}
            \item \cite{Palumbo2013} In the Hechler extension over $V$, for every real $x \in 2^\omega \setminus V$, there is $c \in V[x]$ which is a Cohen real over $V$.
            \item \cite{Truss1977} Let $W$ be a forcing extension of $V$. If there is a Cohen real $c$ in $W$ over $V$ and there is a dominating real in $W$ over $V[c]$, then $W \models \bigcup (\meager \cap V) \in \meager$.
            \item \cite{Pawlikowski1986} $\Hechler \forces \bigcup (\meager \cap V) \not \in \meager$.
        \end{enumerate}
    \end{fact}

    \begin{thm}\label{thm:hechlerbmin}
        $\Hechler \forces V[\name{d}] \text{ is a b-minimal extension over } V$.
    \end{thm}
    \begin{proof}
        Let $d$ be a Hechler real over $V$ and work in $V[d]$.
        Let $z \in \omega^\omega$.
        Let $x$ be a real that does not belong to $V$. We may assume that $x$ is a Cohen real over $V$ using Fact \ref{fact:hechler} (1).
        In order to get a contradiction, we assume $z$ is a dominating real over $V[x]$.
        Then, by Fact \ref{fact:hechler} (2), we have $\bigcup (\meager \cap V) \in \meager$ in $V[d]$.
        This contradicts Fact \ref{fact:hechler} (3).
    \end{proof}

    \section{An example of a d-minimal, weakly $\omega^\omega$-bounding extension that is neither minimal nor $\omega^\omega$-bounding}\label{sec:example}
    
    We consider the extension $V[m][s]$, where $m$ is Miller over $V$, and $s$ is Sacks over $V[m]$.

   \begin{thm}\label{thm:millersacks}
       If $x\in V[m,s] \cap 2^\omega$, then one of the following must be true (where $V$, $V[x]$ etc are to be read as $V\cap 2^\omega$, etc.)
   \begin{enumerate}
       \item $V[x]=V$
       \item $V[x]=V[m]$
       \item $V[x]=V[m][s]$
   \end{enumerate}
   \end{thm}

   \begin{defi}
       Let $P$ be a forcing notion, $\name x$ a $P$-name, $p\in P$, and $p \forces \name x : \omega \to \omega$. 
       We say that $y:\omega\to \omega$ is an \emph{interpretation} of $\name x$ below $p$ iff there is a sequence $\cdots \le p_2\le p_1\le p_0\le p$ such that for all $n\in \omega$ we have $p_n\forces \name x\on n = y\on n$. 
   \end{defi}
   \begin{lem}[Divergence]\label{lem:divergence}
   Let $P$ be a forcing notion, $\name x$ a $P$-name of an element of $2^\omega$, and  $\forces_{P} \name x \notin V$. 
       Assume that $\seq{p_n: n \in \omega}$ is a family of conditions. Then there are reals $x_n\in 2 ^\omega$, all distinct, and each $x_n$ is an interpretation of $\name x$ below $p_n$. 

       Moreover, we can find an infinite subset $W\subseteq \omega$ and conditions $\seq{p'_n:n\in W}$, $p'_n\le p_n$, and finite sequences $\eta_n\in 2^{<\omega}$ for $n\in W$, pairwise incompatible, such that ($\eta_n\subseq x_n$ and) $p'_n\forces \eta_n\subseq \name x$.
    
   \end{lem}
    \begin{proof}
        Part 1 is done by induction.  When we are already given interpretations $\seq{x_k:k<n}$, we can finitely often strengthen the condition $p_n$ to arrive at a condition $p_n'$ which for each $k<n$ knows some index $i_k$ and forces $\name x(i_k)\not= x_k(i_k)$. Now find an interpretation below $p_n'$.

        To show the ``moreover'' part: By thinning out our family, we may assume that the reals $x_n$ converge to a real $y$, but are all different from $y$, say $k_n$ is minimal with $x_n\on k_n \not= y\on k_n$. We may assume $n<n'\Rightarrow k_n<k_{n'}$. Then the finite sequences $\eta_n:=x_n\on k_n$ are pairwise incompatible, since for $n<n'$ we have $y\on k_n \perp x_n\on k_n$, but $y\on k_n \subseq x_{n'}$. 
        
        Now, use the interpretation below $p'_n$ to get a stronger condition $p_n''$ which   forces $x_n \on k_n\subseq  \name x $.
     \end{proof}
     
     \begin{lem}[Total divergence]\label{lem:totaldivergence} 
     Let $K$ be a finite set. 
         Assume that $\seq{\eta_{n,k}:n\in \omega, k\in K}$ is a family in $2^{<\omega}$ such that 
         \begin{itemize}
             \item  For each $k\in K$, the family $\seq{\eta_{n,k}:n\in \omega}$ is pairwise incompatible.          
         \end{itemize}
         Then we can find an infinite subset $N\subseteq \omega$ such that the family 
     $\seq{\eta_{n,k}:n\in N, k\in K}$ is pairwise incompatible.
     \end{lem}
     \begin{proof} (Easy combinatorics using Ramsey's theorem.)  After thinning out we may assume that for each $k\in K$ the sequence $\seq{|\eta_{n,k}|:n\in \omega}$ is strictly increasing.  After thinning out again we can get
      $n<n' \Rightarrow |\eta_{n,k}| < |\eta_{n',k'}| $ for all $k,k'\in K$.
      
       We color each pair $(n,n')$ (with $n<n'$) by the set $C_{n,n'}:=\{(k,k'): \eta_{n,k} \le \eta_{n',k'}\} \subseteq K\times K$; by Ramsey's   theorem we can find a fixed set $C$ and an infinite set $N\subseteq \omega$  such that the coloring is constant on $[N]^2$, i.e., for all $n<n'$ in $N$ we have $C_{n,n'}=C$. 

       Now let $(n,k),(n',k')$ be such that $n<n'$ are in $N$, and $\eta_{n,k}, \eta_{n',k'}$ are compatible, so $\eta_{n,k}\le \eta_{n',k'}$. We will show that this leads to a contradiction.

       We have $(k,k')\in C$. Let $n''>n'$ be in $N$. So 
       \begin{itemize}
           \item From $n,n''\in N$ we get $C_{n,n''}=C$, so $(k,k')\in C_{n,n''}$, hence  $\eta_{n,k} \subseq \eta_{n'',k'}$.
           \item Similarly: from $n',n''\in N$ we get  $\eta_{n',k} \subseq \eta_{n'',k'} $.
       \end{itemize}
       Hence both $\eta_{n,k}$ and $\eta_{n',k}$ are initial segments of $\eta_{n'',k'}$; this is impossible, as they are (by assumption) incompatible.
     \end{proof}
\newcommand{\xxx}{\textcolor{red}{!}}


\begin{defi}
     For a Miller or Sacks condition $q$, let $\splt(q)$ be the set of splitting points of $q$. For $n\in \omega$ let $\splt_n(q):= \{\nu\in \splt(q): |\{ \nu'\in \splt(q): \nu'\subseqneq \nu\}|=n\}$  be the $n$-th splitting front of~$q$. 

     For $\eta\in q$ we write $q^\eta$ for the set of all nodes in $q$ which are comparable with $\eta$. 

     We write $q'\le_n q$ iff $q'\le q$ and $\splt_n(q')=\splt_n(q)$.

     If $q$ is a Sacks condition, $n\in \omega$, $\onetotwo$, then we write $q(n,i)$ for the $i$-th element of $\splt_{n+1}(q)$ (in the lexicographic ordering).  We write $q^{(n,i)}$ for $q^{q(n,i)}$.

     

\end{defi}

     \begin{lem}[Fusion step 0]\label{lem:fusion0}
     Assume $(p,\name q)\in \Miller*\name \Sacks$, $n\in \omega$, $(p,\name q)\forces \name x\notin V[m]$, 
     $n\in \omega$.

         Then we can find a condition $p'\le_0 p$, $S:=\suc_{p'}(\stem(p))$, a name $\name q'$ such that $p'\forces \name q'\le_n \name q $, and 
         a family 
         $\seq{\eta_{s,i}:s\in S, \onetotwo}\in V$, such that 
         \begin{itemize}
         \item for all $s\in S$, $\onetotwo$: 
         $(p^s, \name q'^{(n,i) }) \forces \eta_{s,i} \subseq \name x $, 
         \item for each $i$ the family $\seq{\eta_{s,i}:s\in S}$ is pairwise incompatible.
         \item 
         Moreover, the whole family $\seq{\eta_{s,i}:s\in S, \onetotwo }$ is pairwise incompatible.
         \end{itemize}
         
     \end{lem}
     \begin{proof}
     Let $r:=\stem(p)$. 
     
         We will find $p'$ in $2^{n+1}$ steps. 
         \begin{itemize}
             \item 
         In the first step we deal with $i=1$. Let $S_0:=\suc_p(r)$.\\
         Apply  Lemma~\ref{lem:divergence} to get an infinite subset $S_1\subseteq S_0$, a family $\seq{(p_{s,1}, \name q_{s,1}):s\in S_1}$ with $p_{s,1}\le p^s$, $p_{s,1}\forces \name q_{s,1} \le \name q^{(n,1) }$ and a pairwise incompatible family $\seq{\eta_{s,1}:s\in S}\in V $ with $(p_{s,1}, \name q_{s,1}) \forces \eta_{s,1}\le \name x$. 

         As the family $\seq{p_s:s\in S}$ is an antichain, we can join or ``mix'' these conditions to get a single condition $p_1\le_0 p$ and a name $\name q_1 $ such that $p_1^s \forces \name q_1^{({s,1})} = \name q_{s,1}$. We define the remainder of $\name q_1$ by demanding $\name q_1^{({s,i})} = \name q^{({s,i})}$ for all $s$ and all $i\not=1$. \\
         As $\name q_1$ is forced  to contain all nodes of the form $\name q(n,i)$, we have $p_1\forces \name q_1\le_n \name q$.

         \item In the second step we deal with $i=2$. \\
         Apply Lemma~\ref{lem:divergence} to the family
         $\seq{(p_1^s, q^{({s,2})}):s\in S_1}$ to get a condition $(p_2, \name q_2)$ and a pairwise incompatible family $\seq{\eta_{s,2}:s\in S_2}$, where $S_2:=\suc_{p_2}(r)$.\\
         The condition $\name q_2$ will again satisfy $\splt_n(\name q_2)=\splt_n(\name q)$, 
         and $p_2^s \forces \name q_2^{\nu_{s,i} }= \name q_1^{\nu_{s,i}}$ for all $i\not=2$.
         \item etc.
         \end{itemize}
         After $2^{n+1}$ steps we obtain the desired $p':=p_{2^{n+1}}$, with $S:=\suc_{p'}(r)$.  For each $s\in S$ the condition $\name q'$ 
         will satisfy $p' \forces \name q^{\prime({s,i})} = \name q_i^{(s,i)}$. 

         By construction, for each $i$ the family $\seq{\nu_{s,i}:s\in S}$ is pairwise incompatible.
         

         Applying Lemma~\ref{lem:totaldivergence} we get a subfamily $\seq{\nu_{s,i}:s\in S', \onetotwo}$ which is pairwise incompatible. 
     \end{proof}
     
 \begin{lem}[Fusion step $n$]\label{lem:fusionn}
     Assume $(p,\name q)\in \Miller*\name \Sacks$, $n\in \omega$, $(p,\name q)\forces \name x\notin V[m]$, and let $R:=\splt_n(p)$.  Assume also that
     $\seq{\eta_{r,i}: r\in R, \onetotwo }$ 
     is a family  of elements of $2^{<\omega}$, such that:
      \begin{itemize}
         \item for all $r\in R$, $\onetotwo$: 
         $(p^r,  \name q^{(n,i)})  \forces \eta_{r,i} \subseq \name x $.
      \item the family $\seq{\eta_{r,i}: r\in R, \onetotwo }$ is pairwise incompatible.
              \end{itemize}
         Then we can find a condition $(p',\name q')\le_n (p,\name q)$, $S:=\splt_{n+1}(p')$,\\ and a family 
         $\seq{\eta_{s,i}:s\in S, \onetotwo}\in V$, such that: 
         \begin{itemize}
         \item for all $s\in S$, $\onetotwo$: 
         $(p^{\prime s}, \name q^{(s,i)})\forces \eta_{s,i} \subseq \name x $
         \item the family $\seq{\eta_{s,i}:s\in S,\onetotwox}$ is pairwise incompatible.
         \end{itemize}

\end{lem}
   \begin{proof}
       Apply \ref{lem:fusion0} to each element $r$ of $\splt_n(p)$ (separately).  For each $r$ we thus get a set $S \subseteq \suc_p(r)$, and appropriate families $\seq{\name q_s:s\in S_r}$ and $\seq{\eta_{s,i}:s\in S_r,\onetotwox}$.  Together, the conditions $\seq{p_s: s\in S_r, r\in \splt_n(p)}$ define a new condition $p'\le_n p$, and the names $\name q_s$ can be mixed together to give a name $\name q'$. 

       For each $r$, the family $\seq{\eta_{s,i}:s\in S_r, \onetotwox}$ is pairwise incompatible. We also claim that for  $r_1\not= r_2$, $s_1\in S_{r_1}$, $s_2\in S_{r_2}$, $\onetotwoy{i_1,i_2}$ the sequences 
       $\eta_{s_1,i_1}$ and $\eta_{s_2,i_2}$ are incompatible.  To prove this, 
       note that there is there are  $j_1,j_2$ such that 
       $\eta_{r_1,j_1} \subseteq \eta_{s_1,i_1} $
       $\eta_{r_2,j_2} \subseteq \eta_{s_2,i_2} $, 
       and $\eta_{r_1,j_1} , \eta_{r_2,j_2} $ are incompatible, as $r_1\not=r_2$. 
          \end{proof}      
\begin{proof}[Proof of Theorem~\ref{thm:millersacks}]

Assume that $\name x$ is a $\Miller*\name \Sacks$-name, $(p,\name q)\forces \name x\in 2^\omega$. 
If we can find a stronger condition that forces $\name x\in V[\name m]$ (where $\name m$ is the canonical name of the Miller real), then this condition also forces $V[\name x] =  V \ \vee \ V[\name x]\cap 2^\omega = V[\name m]$, as Miller forcing adds a minimal real. 

So let us assume that $(p,\name q) \forces \name x\notin V[\name m]$. Our goal is to find a stronger condition that forces $\name x$ to generate the full extension. 

We will use Lemma~\ref{lem:fusionn} to  construct a fusion sequence $(p_n,\name q_n)$ of conditions. Letting $(p_\infty, \name q_\infty)$ be their limit, we claim that $(p_\infty, \name q_\infty) \forces \name m\in V[\name x]$. 

Let $G=G_\Miller* G_\Sacks $ be a generic filter containing $(p_\infty, \name q_\infty)$.  Let $m:=\name m[G_\Miller]$, $x:=\name x[G]$. Work in $V[G]$.   
\begin{itemize}
    \item 
Let $r_0:=\stem(p)$, $S_0:=\suc_p(r_0)$.  Consider the family $\seq{\eta_{s,i}:s\in S_0, i\in \{1,2^1\}}$: there must be a unique pair $(s_0,i_0)$ such that $\eta_{s_0,i_0} \subseq x$.  So we conclude that $r_0^\frown s_0 \subseteq m$. 
   \item Let $r_1:= \stem(p^{r_0 \append s_0})$. We already know $r_1 \subseteq m$. Let $S_1:= \suc_p(r_1)$. From the family $\seq{\eta_{s,i}: s\in S_1, i\in \{1,\ldots, 2^2\}}$ we find a unique member $\eta_{s_1, i_1} $ which is extended by $x$, so we conclude that $m $ extends $r_1 \append s_1 $. 
  \item etc   
\end{itemize}


This shows $m\in V[x]$.  
As $x\notin V[m]$, we get $V[m] \subsetneq V[x]$. 
 As Sacks forcing is minimal we conclude   $V[x] = V[G]$. 
   
\end{proof}

%
%

     \section{Constellations}\label{sec:constellations}

    This section provides examples of all constellations that assign ``true" or ``false" to each vertex in the following diagram.

    
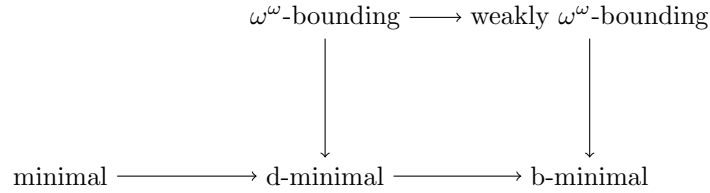
\[
	\begin{tikzpicture}[scale=0.7]
        \node (minimal) at (0, 0) {minimal};
        \node (dminimal) at (5, 0) {d-minimal};
        \node (bminimal) at (10, 0) {b-minimal};
        \node (notunbdd) at (5, 3) {$\omega^\omega$-bounding};
        \node (notdominating) at (10, 3) {weakly $\omega^\omega$-bounding};
        \draw[->] (minimal) -- (dminimal);
        \draw[->] (dminimal) -- (bminimal);
        \draw[->] (notunbdd) -- (notdominating);
        \draw[->] (notunbdd) -- (dminimal);
        \draw[->] (notdominating) -- (bminimal);
	\end{tikzpicture}
    \]

    In the following, $\blacksquare$ and $\square$ mean true and false, respectively.
    
    \subsection{The Sacks extension}
    
    \diagram{minimal=1, dminimal=1, bminimal=1, notunbdd=1, notdominating=1}

    The Sacks extension satisfies the above diagram, since it is minimal \cite{Sacks1971PerfectClosedSets} and is $\omega^\omega$-bounding.
    
    \subsection{The Miller extension}
    
    
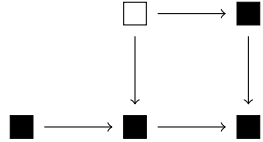
\diagram{minimal=1, dminimal=1, bminimal=1, notunbdd=0, notdominating=1}

    The Miller extension satisfies the above diagram, since it is minimal \cite{Miller1984RationalPerfectSet} and it adds an unbounded real, but no dominating real.
    
    \subsection{The Laver extension}
    
    
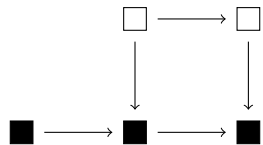
\diagram{minimal=1, dminimal=1, bminimal=1, notunbdd=0, notdominating=0}

    The Laver extension satisfies the above diagram, since it is minimal \cite{Gray1980IteratedForcing} and adds a dominating real.

    \subsection{The random extension}

    
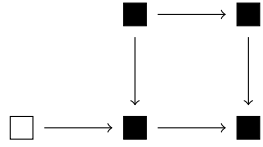
\diagram{minimal=0, dminimal=1, bminimal=1, notunbdd=1, notdominating=1}

    The random extension satisfies the above diagram, since it is $\omega^\omega$-bounding and it is not minimal: the random real can be split into two parts which are both random reals and neither of them can compute the other one.

    \subsection{The extension from Section \ref{sec:example}}

    \diagram{minimal=0, dminimal=1, bminimal=1, notunbdd=0, notdominating=1}
    
    The extension from Section \ref{sec:example} satisfies the above diagram. To see that it is $d$-minimal, note that we showed in section~\ref{sec:example} the only properly intermediate model generated by a real is $V[m]$, and $V[m,s]$ is $\omega^\omega$-bounding over this intermediate model. 
    
    \subsection{The Mathias extension}

    
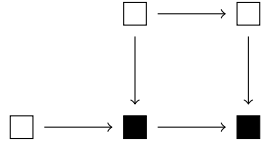
\diagram{minimal=0, dminimal=1, bminimal=1, notunbdd=0, notdominating=0}

    The Mathias extension satisfies the above diagram, due to the result in Section \ref{sec:mathias}.
    The Mathias extension is not minimal: its witness is any infinite  coinfinite subset of the Mathias real $m$ (see Corollary 4.10(ii) and Theorem 8.2 of \cite{Mathias1977HappyFamilies}).

    \subsection{The Cohen extension}

    
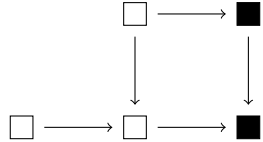
\diagram{minimal=0, dminimal=0, bminimal=1, notunbdd=0, notdominating=1}

    The Cohen extension satisfies the above diagram.
    It is not d-minimal: its witness is the even part of the Cohen real $c$ due to mutually genericity of the even part and the odd part. Also, the Cohen forcing does not add a dominating real.
    
    \subsection{The Hechler extension}

    
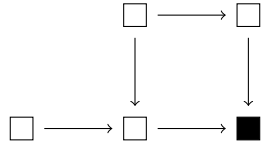
\diagram{minimal=0, dminimal=0, bminimal=1, notunbdd=0, notdominating=0}

    The Hechler extension satisfies the above diagram, mainly due to the result in Section \ref{sec:hechler}.
    Hechler extension is not d-minimal: let $d$ be the Hechler real and let $c(n) := d(n) \bmod 2$. This $c$ is a Cohen real over $V$ and the enumeration of $\{ n : c(n) = 1 \}$ is infinitely often below $d$.
    
    \subsection{The localization extension}

    
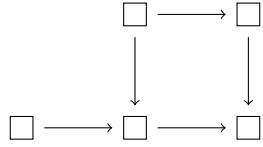
\diagram{minimal=0, dminimal=0, bminimal=0, notunbdd=0, notdominating=0}

    The localization forcing $\mathbb{LOC}$ is a ccc forcing that adds a slalom $\varphi \in ([\omega]^{<\omega})^\omega$ such that $|\varphi(n)| \le n$ for every $n$ and $f(n) \in \varphi(n)$ for almost all $n$ for every $f \in \omega^\omega$ in the ground model (for the definition, see p.106 of \cite{BartoszynskiJudah1995}).

    We show the localization extension is not b-minimal.

    Let us define a real $x$ as follows:
    \[
    x(n) := \sum_{a \in \varphi(n)} 2^a,
    \]
    where $\varphi$ is the generic slalom.

    Fix a random real $r$ over $V$ in $V[x] = V[\varphi]$, which exists.
    Fix $y \in \omega^\omega \cap V[r]$.
    Since the random forcing is $\omega^\omega$-bounding, there is $z \in V$ such that $y < z$. By the genericity, we have $z(n) \in \varphi(n)$ for almost all $n$, which implies  $z(n) < x(n)$ for almost all $n$. Therefore, $y <^* x$ holds.
    
    \section{Discussion}

    To state the open problems, we define two notions.
    Let $W$ be a forcing extension of $V$.
    In $W$, we define the following notions.
        \begin{itemize}
            \item $g \in W \cap \omega^\omega$ is d-minimal over $V$ iff:\\ for every $x \in 2^\omega \setminus V$, there is $y \in V[x] \cap \omega^\omega$ such that $g <^* y$.
            \item $g \in W \cap \omega^\omega$ is b-minimal over $V$ iff:\\ for every $x \in 2^\omega \setminus V$, there is $y \in V[x] \cap \omega^\omega$ such that $g <^\infty y$.
        \end{itemize}

    The following problems remain.

    \begin{question}
        \begin{enumerate}
            \item Is it true that $\Hechler \ast \name{\Hechler} \forces \name{d}_0 \text{ is a b-minimal real over } V$? (Here, $\name{d}_0$ is the first Hechler real).
            \item Is it true that $\Mathias \ast \name{\Mathias} \forces \name{m}_0 \text{ is a d-minimal real over } V$? (Here, $\name{m}_0$ is the first Mathias real). 
            \item How about longer iterations?
        \end{enumerate}
    \end{question}

    A positive answer to the question about longer iterations, in the following strengthened form, would yield a descriptive-set-theoretic characterization of the reals in intermediate stages. Let $G_\kappa$ be generic for a finite-support iteration of Hechler forcing or a countable-support iteration of Mathias forcing, and let $r_\alpha$ denote the dominating real added at stage $\alpha<\kappa$. Suppose that, in $V[G_\kappa]$, for every $\alpha<\kappa$ and every $x\in 2^\omega\setminus V[G_\alpha]$, there is $y\in V[x]\cap\omega^\omega$ such that $r_\alpha<^\infty y$ (or, more strongly, $r_\alpha<^*y$ in the Mathias case). Then
    \[
    V[G_\alpha]\cap 2^\omega
    =
    \left\{x\in 2^\omega:
    \forall y\in V[x]\cap\omega^\omega\;\; y<^*r_\alpha
    \right\}.
    \]
    Indeed, if $x\in V[G_\alpha]$, then $V[x]\subseteq V[G_\alpha]$, and $r_\alpha$ dominates every function in $V[G_\alpha]\cap\omega^\omega$. The converse follows from the assumed property.
    
    If the ground model satisfies $V=L$, this gives a $\boldpi^1_2$ definition of $V[G_\alpha]\cap 2^\omega$ with the single real parameter $r_\alpha$.


    \section*{Use of AI}

	The authors used ChatGPT during the preparation of this paper as a source of heuristic suggestions concerning some of the proofs. All mathematical arguments were independently verified and written by the authors, who take full responsibility for the contents of the paper.

	\printbibliography
\end{document}